\documentclass[12pt]{amsart}
\usepackage{latexsym}
\usepackage{amssymb, amsmath,mathrsfs, mathtools}
\usepackage[left=2.5cm,right=2.5cm,top=2.5cm, bottom=2.5cm ]{geometry}
\usepackage{todonotes}
\usepackage{array}
\usepackage{setspace}
\usepackage{graphicx, subfigure}
\usepackage{caption}
\usepackage{tikz}
\usepackage{pgfplots}
\pgfplotsset{compat=1.18}
\usepgfplotslibrary{fillbetween}
\newcommand\beq{\begin{equation}}

\newcommand\eeq{\end{equation}}

\usepackage{float}
\makeatletter
\@namedef{subjclassname@2020}{%
  \textup{2020} Mathematics Subject Classification}

\usepackage[pagebackref,hypertexnames=false, colorlinks, citecolor=red, linkcolor=red]{hyperref}

\newcommand{\McC}{\raise.5ex\hbox{c}}

\newcommand{\D}{\mathbb{D}}

\newcommand\bbm{\begin{bmatrix}}
\newcommand\ebm{\end{bmatrix}}

\newtheorem{theorem}{Theorem}[section]

\newtheorem{lemma}[theorem]{Lemma}

\newtheorem*{theorem*}{Theorem}
\newtheorem*{conjecture*}{Conjecture}
\newtheorem{corollary}[theorem]{Corollary}
\newtheorem*{corollary*}{Corollary}

\newtheorem*{proposition*}{Proposition}

\def\bb{\begin{color}{blue}}
\def\bg{\begin{color}{green}}
\def\br{\begin{color}{red}}
\def\eg{\end{color}}
\def\er{\end{color}}
\def\eb{\end{color}}

\theoremstyle{remark}
\newtheorem{remark}[theorem]{Remark}
\newtheorem{question}{Question}
\newtheorem{definition}[theorem]{Definition}
\newtheorem{example}[theorem]{Example}

\author[Bickel]{Kelly Bickel$^\dagger$}
\address{Department of Mathematics \& Statistics, Bucknell University, Lewisburg, PA 17837, USA.}
\email{kelly.bickel@bucknell.edu}
\thanks{$\dagger$ Research supported in part by National Science Foundation DMS grant \#2000088.}

\author[Pascoe]{J. E. Pascoe$^\ddagger$}
\address{Department of Mathematics, Drexel University, Philadelphia, PA 19104, USA}
\email{jep362@drexel.edu}
\thanks{$\ddagger$ Research supported in part by National Science Foundation DMS grant \#1953963.}

\author[Tully-Doyle]{Ryan Tully-Doyle$^*$}
\address{Department of Mathematics, Cal Poly SLO, 1 Grand Ave, San Luis Obispo, CA 93410, USA}
\email{rtullydo@calpoly.edu}
\thanks{$*$ Research supported in part by National Science Foundation DMS grant \#2055098}

\begin{document}
\title[Automorphic Nelson Dilation]{Automorphic Nelson Dilations for Contractions and Invariant Subspace Tracking}
\date{\today}
\begin{abstract}
Given an $n \times n$ strictly contractive matrix $T$, an (automorphic) Nelson dilation $\widehat{T}$ of $T$ is a certain type of analytic matrix-valued function on the unit disk with $\widehat{T}(0) = T$.  Its construction gives a method for lifting a matrix to a matrix-valued function with nice boundary behavior, a trick that has proved useful in recent operator theoretic developments. In this paper, we show that Nelson dilations give a quick way to obtain the minimal isometric and unitary dilations of $T$ and thus, connect naturally to the classical Sz.-Nagy dilation theory. We then initiate the study of the automorphic Nelson dilations as a fundamental object in their own right and prove that every $T$  has Nelson dilations $\widehat{T}$ with particularly useful/interesting properties; for example, they either have strongly entangled eigenvalue functions or have reducing subspaces that are independent of $z$.  Along the way, we examine when the product of an invertible matrix and a diagonal matrix has distinct eigenvalues.
\end{abstract}
\maketitle


\section{Introduction} \label{sec:intro}

Let $T$ be an $n \times n$ matrix with $\| T\| <1$. Then $T$ is a strict contraction on $\mathbb{C}^n$ and has a singular value decomposition 
\begin{equation} \label{eqn:svd} T = U \begin{bmatrix} s_1 && \\
& \ddots & \\
&& s_n \end{bmatrix}  V^{*},\end{equation}
with $U, V$ unitary and (phased) singular values $s_1, \dots, s_n$ in the unit disk $\mathbb{D}$. Given decomposition \eqref{eqn:svd}, the associated  \textbf{classical Nelson dilation of $T$} is the matrix-valued function 
 \begin{equation} \label{eqn:cnd} \widetilde{T}(z_1, \dots, z_n):= U \begin{bmatrix} z_1 && \\
& \ddots & \\
&& z_n \end{bmatrix}  V^{*}, \end{equation}
where the singular values are replaced by independent complex variables $z_1, \dots, z_n \in \overline{\mathbb{D}}$ and $\widetilde{T}(s_1, \dots, s_n ) = T$.  Such Nelson dilations have proved to be an important tool in the development of both dilation theory and more general operator theory pursuits. Indeed, Nelson used these dilations in \cite{nelson} to
 provide a short and elegant proof of von Neumann's inequality. Nelson's proof techniques also motivated recent work by Hartz in \cite{hartz} 
 on a multivariable version of von Neumann's inequality for weighted shifts. 

Now let us consider a related type of dilation of $T$. Specifically, let $b_1, \dots, b_n$ be automorphisms of $\mathbb{D}$ (then, each $b_i(z)  = \lambda_i \left(\frac{z+a_i}{1+ \bar{a}_i z}\right)$ for some $\lambda_i$ in the unit circle $\mathbb{T}$ and $a_i \in \mathbb{D}$) with $b_i(0) = s_i$ for $i =1, \dots, n$. Then the matrix-valued function 
 \begin{equation} \label{eqn:nd} \widehat{T}(z):= U \begin{bmatrix} b_{1}(z) && \\
& \ddots & \\
&& b_n(z) \end{bmatrix}  V^{*} \end{equation}
 satisfies $\widehat{T}(0) = T$. We will call $\widehat{T}(z)$ an (automorphic) \textbf{Nelson dilation of $T$}.  In \cite{mp}, the second author and McCullough introduced and used this automorphic version of Nelson's trick to study dilation  theory for contractions, the quantum annulus, and many other domains. The general principle at work in \cite{ hartz, mp, nelson} is to extend boundary estimates on an analytic family with good boundary values to the interior by invoking some kind of maximum principle. Historically, Paulsen gives a finite dimensional version of Nelson's argument in \cite[Exercise 2.16]{paulsen} and attributes it to Wermer. Pisier uses the argument in \cite{pisier} to prove von Neumman's inequality and attributes it to Nelson via communication with Paulsen and Wermer.

In this paper, we identify and develop core properties of (automorphic) Nelson dilations as in \eqref{eqn:nd}, with a particular interest in the relationship with the classical theory of isometric/unitary dilations and the analytic properties of the resulting family of matrices $\widehat{T}(z)$.  To connect Nelson dilations with the  classical dilation theory for Hilbert space contractions, we need several key definitions. First, let $S$ be a bounded linear operator on a  Hilbert space $\mathcal{H}$, denoted $S \in B(\mathcal{H})$, with $\| S\| \le 1$. Then a unitary (resp. isometry) $V$ on a Hilbert space $\mathcal{K}$ is a \textbf{unitary} (resp. isometric) \textbf{dilation}\footnote{In the literature, this is sometimes called a unitary (resp. isometric) \textbf{power dilation} of $S$.} of $S$ if $\mathcal{H} \subseteq \mathcal{K}$ and
\begin{equation} \label{eqn:dilation} S^k = P_{\mathcal{H}} V^k |_{\mathcal{H}} \quad \text{ for all } k \in \mathbb{N}.\end{equation}
An isometric dilation $V\in B(\mathcal{K})$ of a contraction $S \in B(\mathcal{H})$ is called \textbf{minimal} if $\mathcal{K}$ is the minimal, closed invariant subspace of $V$ that contains $\mathcal{H}$, namely:
\[ \mathcal{K} =  \overline{ \text{span}} \left \{ V^m h : h \in \mathcal{H} , m \in \mathbb{N} \right\}.\] 
Meanwhile, a unitary dilation $V\in B(\mathcal{K})$ of a contraction $S \in B(\mathcal{H})$ is called \textbf{minimal} if $\mathcal{K}$ is the minimal, reducing subspace of $V$ (closed and invariant under both $V$ and $V^*$, which in this case equals $V^{-1}$) that contains $\mathcal{H}$, namely:
\[ \mathcal{K} =  \overline{ \text{span}} \left \{ V^m h : h \in \mathcal{H} , m \in \mathbb{Z} \right\}. \] 
The classical work of Sz.-Nagy \cite{nagy53, nagyfoias2010}  identified the unique (up to unitary equivalence) minimal isometric and unitary dilations of a general contraction $S$ and has played a crucial role in the development of one and multi-variable operator theory. See the readable survey paper \cite{shalit21} and comprehensive book \cite{nagyfoias2010} for more information about this beautiful theory. Halmos also considered the unitary dilation problem in \cite{halmos}, where he gives an argument via a $2 \times 2$ block matrix completion.  

In Section \ref{minimality} of this paper, we observe that for an $n\times n$ strictly contractive matrix $T$,  the Nelson dilations of $T$ give rise to concrete and natural formulas for minimal isometric and unitary dilations of $T$. To state the result, 
let $H^2(\mathbb{D})$ denote the standard Hardy space on unit disk $\mathbb{D}$, let $L^2(\mathbb{T})$ denote the standard $L^2$ space on the unit circle $\mathbb{T}$, let $H^2_n$ denote the vector-valued Hardy space $H^2_n:=H^2(\mathbb{D}) \otimes \mathbb{C}^n$, and let ${L_n^2}$ denote the vector-valued  $L^2$ space $L^2_n:=L^2(\mathbb{T}) \otimes \mathbb{C}^n$. A direct study of the Nelson dilation construction gives the following:

\begin{theorem} \label{thm:minimal} Let $T$ be an $n\times n$ matrix with $\| T\|<1$ and let $M_{\widehat{T}}$ denote multiplication by a Nelson dilation $\widehat{T}$ of $T$ as in \eqref{eqn:nd}. Then, $M_{\widehat{T}}|_{H_n^2}$ is a minimal isometric dilation of $T$ and $M_{\widehat{T}}|_{L_n^2}$  is a minimal unitary dilation of $T.$
\end{theorem}

Meanwhile, in Section \ref{fractured}, we study specific types of Nelson dilations. In particular, given a Nelson dilation $\widehat{T}$ as in \eqref{eqn:nd}, let $\mathbb{D}_T$ denote the largest open disk centered at $0$ where each  automorphism $b_i$ is analytic. Then $\overline{\mathbb{D}} \subseteq \mathbb{D}_T$ and the entries of $\widehat{T}$ are all analytic on $\mathbb{D}_T$.  The set of eigenvalues of $\widehat{T}(z)$, denoted $\sigma(\widehat{T}(z))$,  are the solutions to the algebraic equation $\det ( x I - \widehat{T}(z))=0$ and thus, form an $n$-valued, analytic function on $\mathbb{D}_T$. On neighborhoods where the eigenvalues are distinct,
$\sigma(\widehat{T}(z))$ can be parameterized by $n$ separately-analytic functions, but when eigenvalues come together, more
 pathological eigenvalue/eigenvector behavior can occur. Thus we define the \textbf{exceptional set} of $\widehat{T}$, denoted $S_{\widehat{T}}$, to be 
 \begin{equation} \label{eqn:exceptional} S_{\widehat{T}} := \left \{ z \in \mathbb{D}_T: \widehat{T}(z) \text{ has a repeated eigenvalue} \right\}\end{equation}
and say that a Nelson dilation $\widehat{T}$ is \textbf{fractured} if  $S_{\widehat{T}}$ is finite. In Section \ref{fractured}, we prove the following:

\begin{theorem} \label{thm:fractured} Every square matrix $T$ with $\| T \| <1$ has a fractured Nelson dilation $\widehat{T}$. 
\end{theorem}

Thus, when studying a square matrix $T$ with $\| T\|<1$, one can always find a Nelson dilation $\widehat{T}$ whose exceptional set is finite. Such a fractured Nelson dilation $\widehat{T}$ always possesses at least one of the following nice properties: its eigenvalues (as an $n$-valued function of $z$) are entangled in a natural way or there is a nontrivial subspace $S \subseteq \mathbb{C}^n$ that is reducing for 
$\widehat{T}(z)$ for every $z \in \mathbb{D}_T$. This statement is encoded in the following theorem: 

\begin{theorem} \label{thm:entangled} Let $T$ be an $n\times n$ matrix with $\| T \| <1$ and fractured Nelson dilation $\widehat{T}$. Then at least one of the following must occur:
\begin{itemize}
\item[i.] $\widehat{T}$ is \textbf{eigenvalue entangled} in the sense that its eigenvalue function  $\sigma(\widehat{T}(z))$ cannot be partitioned into nontrivial, continuous $k$-valued and $(n-k)$-valued functions $\Lambda_1(z)$ and $\Lambda_2(z)$ such that $\Lambda_1(z) \cap \Lambda_2(z) = \emptyset$ for all $z\in \mathbb{D}_T.$
\item[ii.] $\widehat{T}$ is \textbf{subspace reducible} in the sense that there is a nontrivial subspace $S \subseteq \mathbb{C}^n$ that is reducing for the matrix $\widehat{T}(z)$ for all $z\in \mathbb{D}_T.$
\end{itemize}
\end{theorem} 

If ($i$) fails for a general analytic matrix-valued function $F$, meaning that the eigenvalues of $F$ separate nicely into two functions, then the matrices $F(z)$ should possess nontrivial invariant subspaces $S_z$ that are nicely parameterized by $z$. 
Hence, the noteworthy aspect of Theorem \ref{thm:entangled} is that if a fractured Nelson dilation $\widehat{T}$ has an eigenvalue function that separates nicely, then each matrix $\widehat{T}(z)$ must have a nontrivial invariant subspace $S$ that is both \emph{independent of} $z$ and also \emph{invariant under $\widehat{T}(z)^*$}. 

One can show that properties ($i$) and ($ii$) in Theorem \ref{thm:entangled} are not mutually exclusive, see Example \ref{ex:both}. Thus, in 
Theorem \ref{thm:constantPv}, we also establish a necessary and sufficient condition for $\widehat{T}$ to be subspace reducible called \textbf{projection reducible}, which is discussed in Remark \ref{rem:pr} and contains the negation of \emph{eigenvalue entanged} in ($i$) as a special case.  At the end of Section  \ref{fractured}, we explore these different conditions and results via several examples. 

In Section \ref{sec:fracturing}, we establish the following matrix analysis results, which appear later as Corollaries \ref{cor:MD} and \ref{cor:unitary}. 

\begin{theorem} \label{thm:distinct} Let $M$ be an  $n\times n$ invertible matrix. 
\begin{itemize}
\item[i.] For almost every $n\times n$ diagonal matrix $D$, the matrix $M D$ has distinct eigenvalues. 
\item[ii.] For almost every $n\times n$ diagonal unitary matrix $\Omega$, the matrix $M \Omega$ has distinct eigenvalues. 
\end{itemize}
\end{theorem}

Part ($ii$) of the above theorem is a key tool in the proof of Theorem \ref{thm:fractured}.
Meanwhile, the proof of Theorem \ref{thm:distinct} rests on natural properties of the characteristic polynomials associated to the products $MD$ and $M \Omega$. While it seems possible that Theorem \ref{thm:distinct} is known to the matrix analysis community, we have not been able to locate it in the literature. Instead, it seems very closely related to the recent work of Choi et. al. in \cite{CHLS12}, who affirmed a conjecture posed in \cite{FLH12}  and proved: \emph{if $M$ is an $n \times n$ invertible matrix, then there exists a diagonal matrix $D$ such that $MD$ has distinct eigenvalues}.

The work in this paper also motivates several open questions that the interested reader could explore. We present these in Section \ref{sec:open}. For example, one could study what happens to the location of the exceptional set $S_{\widehat{T}}$ of a fractured Nelson dilation  as the dimension of the underlying matrix $T$ tends to infinity. Results in that vein could allow our tools to give information about invariant subspaces of operators $T$ acting on infinite-dimensional Hilbert spaces. There are also natural open questions connected to the behavior of the eigenvalue functions associated to Nelson dilations and the reducibility of nice multivariate polynomials that appear in Section \ref{sec:fracturing}.

\section{Minimal Dilations} \label{minimality}

Let $T$ be an $n \times n$ matrix satisfying $\| T\| <1$, so that $T$ is a strict contraction on $\mathbb{C}^n$. Let $\widehat{T}$ be a Nelson dilation of $T$ as in \eqref{eqn:nd}. Then the formula for  $\widehat{T}$ immediately implies that $\widehat{T}(\zeta)$ is a unitary matrix for each $\zeta \in \mathbb{T}$. 

To use $\widehat{T}$ to construct dilations of $T$,  we will need  the vector-valued spaces $H^2_n:=H^2(\mathbb{D}) \otimes \mathbb{C}^n$ and $L^2_n:=L^2(\mathbb{T}) \otimes \mathbb{C}^n$ discussed earlier, which satisfy $\mathbb{C}^n \subseteq H^2_n \subseteq L^2_n$. Let  $M_{\widehat{T}}|_{H_n^2}$ denote multiplication by the bounded, analytic  matrix-valued function $\widehat{T}$ on $H^2_n$ and similarly, let $M_{\widehat{T}}|_{L_n^2}$  denote multiplication by $\widehat{T}$ on $L^2_n$. With this setup, we have the following:

\begin{theorem*} \ref{thm:minimal}. $M_{\widehat{T}}|_{H_n^2}$ is a minimal isometric dilation of $T$ and $M_{\widehat{T}}|_{L_n^2}$ is a minimal unitary dilation of $T$.
 \end{theorem*}

\begin{proof} First, we show that  $M_{\widehat{T}}|_{H_n^2}$ is an isometric dilation of $T$ and $M_{\widehat{T}}|_{L_n^2}$ is a unitary dilation of $T$. Indeed, since $\widehat{T}$ is analytic and unitary-valued on $\mathbb{T}$, it is immediate from the standard inner products on $H_n^2$ and $L^2_n$ that $M_{\widehat{T}}|_{H_n^2}$ is an isometry and $M_{\widehat{T}}|_{L_n^2}$ is a unitary.  To see that $M_{\widehat{T}}|_{H_n^2}$ and $M_{\widehat{T}}|_{L_n^2}$ are dilations of $T$, fix $\vec{x} \in \mathbb{C}^n$ and $m \in \mathbb{N}$. Then
\[ P_{\mathbb{C}^n} \ M_{\widehat{T}}^m  \ \vec{x} = P_{\mathbb{C}^n} \ \widehat{T}^m( \cdot ) \  \vec{x} =  \widehat{T}^m(0)  \vec{x} =T^m \vec{x},\]
where $P_{\mathbb{C}^n}$ denotes the orthogonal projection from $H_n^2$ (or $L_n^2$) onto $\mathbb{C}^n$.  Thus, 
\[ T^m  = P_{\mathbb{C}^n} \ M_{\widehat{T}}^m  |_{\mathbb{C}^n},\]
which agrees with \eqref{eqn:dilation} and implies that both $M_{\widehat{T}}|_{H_n^2}$ and $M_{\widehat{T}}|_{L_n^2}$ are dilations of $T$.

For the remainder of the proof, we establish minimality. 
First, we study $H^2_n \ominus M_{\widehat{T}} H^2_n$. Specifically, for $i=1, \dots, n,$ let $b_i(z) =  \lambda_i \left(\frac{z+a_i}{1+ \bar{a}_i z}\right)$ from \eqref{eqn:nd} with $a_i \in \mathbb{D}$ and $\lambda_i \in \mathbb{T}$. Then one-dimensional results (e.g. see \cite[Proposition 5.16]{garciamashreghiross}) imply that $H^2(\mathbb{D}) \ominus b_i H^2(\mathbb{D})$ is spanned by the single function $ v_i(z) = \frac{1}{1+\bar{a_i}z}$. Let $\vec{e}_i$ denote the i$^{th}$ unit vector in $\mathbb{C}^n$ and set $\vec{f}_i =  v_i(z)\vec{e}_i$,   the element in $H^2_n$ that is $v_i(z)$ in the $i^{th}$ position and $0$ otherwise. Then it is straightforward to check that 
\begin{equation} \label{eqn:span} H^2_n \ominus M_{\widehat{T}} H^2_n =\text{span} \left \{ U \vec{f}_i:  1 \le i \le n \right\},\end{equation}
where $U$ is from the singular value decomposition \eqref{eqn:svd} for $T$. Indeed, each $U \vec{f}_i \in H^2_n \ominus M_{\widehat{T}} H^2_n$ because for all $\vec{g} \in H^2_n$,  
\[ \left \langle  U \vec{f}_i, M_{\widehat{T}} \vec{g}\right  \rangle_{H^2_n} = \left \langle   v_i , b_i  (V^* \vec{g})_i \right  \rangle_{H^2} =0,\]
where $ (V^* \vec{g})_i \in H^2(\mathbb{D})$ is the $i^{\text{th}}$ component of the vector-valued function $V^* \vec{g}$. Meanwhile, if $\vec{f} \in H^2_n \ominus M_{\widehat{T}} H^2_n$, then 
\begin{equation} \label{eqn:wandering} \left \langle  \vec{f}, M_{\widehat{T}} \vec{g} \right \rangle_{H^2_n} =  0 \text{ for all } \vec{g} \in H^2_n.\end{equation}
For each $h \in H^2(\mathbb{D})$, choose $\vec{g} = V \vec{e}_i h$. Then \eqref{eqn:wandering} implies  that $(U^* \vec{f})_i \in H^2(\mathbb{D}) \ominus b_i H^2(\mathbb{D})$ and hence is a scalar multiple of $v_i(z)$, as needed. Furthermore, as $\widehat{T}(z)$ is a strict contraction for all $z\in \mathbb{D}$, $\cap_{j \ge 0} M_{\widehat{T}}^j H^2_n  = \{0\}$ and so the Wold decomposition of $M_{\widehat{T}}|_{H_n^2}$ is exactly 
\begin{equation} \label{eqn:wold2} H^2_n = \bigoplus_{j=0}^{\infty}  M_{\widehat{T}}^j \left( H^2_n \ominus M_{\widehat{T}} H^2_n\right).
\end{equation}
Now we need to establish that the $M_{\widehat{T}}$-invariant subspace
\[ S := \overline{ \text{span}} \left \{ \left(M_{\widehat{T}}\right)^m \vec{x} : \vec{x} \in \mathbb{C}^n , m \in \mathbb{N} \right\}\]
  is all of $H^2_n$. 
Note that $S$ clearly contains $\mathbb{C}^n$ and $ \widehat{T}(z) \mathbb{C}^n$. Choose $\vec{x} \in \mathbb{C}^n$ so that $V^*\vec{x} = \vec{e}_i$. Then  $\widehat{T}(z) \vec{x} = U b_i(z) \vec{e}_i \in S$ and by choosing appropriate constants  $c_1, c_2 \in \mathbb{C}$, one can show
\[ U \vec{f}_i =U  \tfrac{1}{1+\bar{a_i}z} \vec{e}_i=  U\left( c_1 + c_2b_i(z) \right) \vec{e}_i =  c_1 U \vec{e}_i + c_2 U b_i(z) \vec{e}_i \in S.\]
By \eqref{eqn:span}, this implies  $H^2_n \ominus M_{\widehat{T}} H^2_n \subseteq S$ and since $S$ is $M_{\widehat{T}}$-invariant, the Wold decomposition of $H^2_n$  in \eqref{eqn:wold2} implies that $S$ must be all of $H^2_n$: 
\begin{equation} \label{eqn:min2} H^2_n := \overline{ \text{span}} \left \{ \left(M_{\widehat{T}}\right)^m \vec{x} : \vec{x} \in \mathbb{C}^n , m \in \mathbb{N} \right\},\end{equation}
 so $M_{\widehat{T}}|_{H^2_n}$ is a minimal isometric dilation of $T$.

 To show $M_{\widehat{T}}|_{L^2_n}$ is a minimal unitary dilation of $T$, we need to show that
 \[ L^2_n = M,  \ \ \text{ where } \ \ M : = \overline{ \text{span}} \left \{ \left(M_{\widehat{T}}\right)^m \vec{x} : \vec{x} \in \mathbb{C}^n , m \in \mathbb{Z} \right\}.\]
 Fix $f\in L^2_n$. Then $ f = f_1 + f_2$ where $f_1 \in H^2_n$ and $f_2(z) = g(\bar{z})$ for some $g\in H^2_n$. By \eqref{eqn:min2}, $f_1 \in M$ and so, we just need to show that $f_2 \in M$. To that end, consider $T^*$. Then from \eqref{eqn:svd}, 
\[ T^* = V \begin{bmatrix} \bar{s}_1 && \\
& \ddots & \\
&& \bar{s}_n \end{bmatrix}  U^{*}.\]
For $i=1, \dots n$, let $b_i(z) = \lambda_i \left( \frac{ z + a_i}{1+\bar{a}_iz}\right)$ be the $i^{\text{th}}$ automorphism for $\widehat{T}$ in \eqref{eqn:nd},  where $\lambda_i \in \mathbb{T}$, $a_i \in \mathbb{D}$, and $\lambda_i a_i = s_i$. Let $\hat{b}_i(z) = \bar{\lambda}_i \left(\frac{ z + \bar{a}_i}{1+a_iz}\right)$. Then using $\hat{b}_1, \dots, \hat{b}_n$ gives a Nelson dilation $\widehat{T^*}$ of $T^*$ satisfying
\[ \widehat{T^*}(\bar{z}) = \widehat{T}(z)^* \text{ for all } z\in \mathbb{T}.\]  
 By the same isometric minimality arguments as before, we can conclude that 
\[ g \in H^2_n = \overline{ \text{span}} \left \{ \left(M_{\widehat{T^*}}\right)^m \vec{x} : \vec{x} \in \mathbb{C}^n , m \in \mathbb{N} \right\}.\]
Then there exists a sequence $(g_k)$ with each 
\[ g_k(z) = \sum_{\ell=0}^{L_k} \widehat{T^*}(z)^\ell \vec{x}_{k \ell}\]
for $L_k \in \mathbb{N}$ and $\vec{x}_{k \ell} \in \mathbb{C}^n$ such that  $|| g_k - g ||_{H^2_n} \rightarrow 0.$ For $z \in \mathbb{T}$, define
\[ h_k(z) = g_k(\bar{z}) =  \sum_{\ell=0}^{L_k}  \widehat{T^*}(\bar{z})^\ell \vec{x}_{k \ell} =  \sum_{\ell=0}^{L_k}  (\widehat{T}(z)^*)^{\ell} \vec{x}_{k \ell} =   \sum_{\ell=0}^{L_k} (\widehat{T}(z))^{-\ell} \vec{x}_{k \ell} \in M. \]
Then $|| h_k - f_2 ||_{L^2_n} = || g_k - g ||_{H^2_n} \rightarrow 0,$ so $L^2_n = M$, as needed.  \end{proof} 


Since minimal isometric (and unitary) dilations are known to be unique up to unitary equivalence, this gives an alternative way to view the classical isometric Sz.-Nagy dilation.

\begin{remark} Specifically, let $T$ be an $n \times n$ matrix with $\| T\| <1$ and let $\widehat{V} \in B(\mathcal{K})$ be the  isometric Sz.-Nagy dilation of $T$. As $T$ is a strict contraction, we can take $\mathcal{K} = \bigoplus_{j=1}^{\infty} \mathbb{C}^n$ and 
\[ \widehat{V} = \begin{pmatrix} 
T & 0 & 0 & 0 & \hdots \\
(1-T^*T)^{1/2}  & 0 & 0 &0 & \hdots  \\
0 & I_{\mathbb{C}^n} & 0 &0 & \hdots  \\
0 & 0 &  I_{\mathbb{C}^n}  &0 & \hdots  \\
\vdots & \vdots & & \ddots & \ddots \\
\end{pmatrix}. \]
This isometric dilation of $T$\footnote{This is sometimes called the Sch\"affer form of a dilation \cite{schaffer55}.}   is known to be minimal. As such, it must be unitarily equivalent to $M_{\widehat{T}}|_{H^2_n}$ via a unitary map 	$W$ that preserves $\mathbb{C}^n$. Indeed, if we first define $W$ via:
\[ W \left( \widehat{V}^m \vec{x} \right) = M_{\widehat{T}}^m \vec{x}, \text{ for all } m \in \mathbb{N}, \vec{x} \in \mathbb{C}^n,\]
then $W$ extends to a unitary map from $\mathcal{K}$ to $H^2_n$ which satisfies $W \widehat{V} = M_{\widehat{T}} W$ and is the identity map on $\mathbb{C}^n$. For details about either this dilation or the unitaries connecting minimal isometric dilations, see  the statement and proof of \cite[Theorem 2.3]{shalit21}.
\end{remark}

We end the section with some examples of Theorem \ref{thm:minimal}.

\begin{example} \label{ex1} Consider the $2\times 2$ matrix $T_1$ with $\| T_1 \| = \frac{1}{2} <1$ and its singular value decomposition:
\[ T_1 = \begin{bmatrix} 0 & \frac{1}{2} \\ 0 & 0 \end{bmatrix} = \begin{bmatrix} 1 & 0 \\ 0 & 1 \end{bmatrix}  \begin{bmatrix} \frac{1}{2} & 0 \\ 0 & 0 \end{bmatrix}  \begin{bmatrix} 0 & 1 \\ 1 & 0 \end{bmatrix}. \]
Then an associated Nelson dilation $\widehat{T}_1$ is given by
\[ \widehat{T}_1(z)  = \begin{bmatrix} 1 & 0 \\ 0 & 1 \end{bmatrix}  \begin{bmatrix} \frac{z+ \frac{1}{2}}{1+ \frac{z}{2}} & 0 \\ 0 & z \end{bmatrix}  \begin{bmatrix} 0 & 1 \\ 1 & 0 \end{bmatrix} =   \begin{bmatrix}0&  \frac{z+ \frac{1}{2}}{1+ \frac{z}{2}}  \\ z & 0 \end{bmatrix} \]
 and so, $M_{\widehat{T}_1}|_{H_2^2}$ is a minimal isometric dilation of $T_1$ and $M_{\widehat{T}_1}|_{L_2^2}$ is a minimal unitary dilation of $T_1$ (with uniqueness up to unitary equivalence).

We can similarly consider the $3\times 3$ matrix $T_2$\footnote{A nice description of $3 x 3$ orthogonal matrices with rational entries appears in \cite{pall}. Up to trivial changes, for the first matrix in $T_2$, we use formula (10) in \cite{pall} with $t_0=t_1=t_2=1$, $t_3=0$, $Nt =3$.} 
 with $\| T_2 \| = \frac{1}{2} <1$ given by 
\[ T_2 = \setstretch{1.7}
\left[
\begin{array}{rrr}
 \frac{i}{6} & -\frac{2i}{15} & -\frac{2}{9} \\
 \frac{i}{3} & \frac{2 i}{15} & -\frac{1}{9} \\
 \frac{i}{3} & -\frac{i}{15} & \frac{2}{9} \\
\end{array}
\right]
=
\left[
\begin{array}{rrr}
 \frac{1}{3} & \frac{2}{3} & \frac{2}{3} \\
 \frac{2}{3} & -\frac{2}{3} & \frac{1}{3} \\
 \frac{2}{3} & \frac{1}{3} & -\frac{2}{3} \\
\end{array}
\right]
\left[
\begin{array}{rrr}
 \frac{i}{2} & 0 & 0 \\
 0 & -\frac{i}{5} & 0 \\
 0 & 0 & -\frac{1}{3} \\
\end{array}
\right]
\left[
\begin{array}{ccc}
1 & 0 & 0 \\
 0 & 1 & 0 \\
 0 & 0 & 1 \\
\end{array}
\right]. \]
Then an associated Nelson dilation $\widehat{T}_2$ is given by
\[ \widehat{T}_2(z) =  \setstretch{2} \left[
\begin{array}{rrr}
 \frac{1}{3} & \frac{2}{3} & \frac{2}{3} \\
 \frac{2}{3} & -\frac{2}{3} & \frac{1}{3} \\
 \frac{2}{3} & \frac{1}{3} & -\frac{2}{3} \\
\end{array}
\right]
\left[
\begin{array}{ccc}
 \frac{z+\frac{i}{2}}{1-\frac{i z}{2}} & 0 & 0 \\
 0 & \frac{z-\frac{i}{5}}{1+\frac{i z}{5}} & 0 \\
 0 & 0 & \frac{z-\frac{1}{3}}{1-\frac{z}{3}} \\
\end{array}
\right] = 
\left[
\begin{array}{rrr}
 \frac{z+\frac{i}{2}}{3-\frac{3i z}{2}} &\frac{2z-\frac{2i}{5}}{3+\frac{3i z}{5}}  & \frac{2z-\frac{2}{3}}{3-z } \\ 
 \frac{2z+i }{3-\frac{3iz}{2}} & -\frac{2z-\frac{2i}{5}}{3+\frac{3i z}{5}} & \frac{z-\frac{1}{3}}{3 -z } \\
 \frac{2z+ i }{3-\frac{3i z}{2}} & \frac{z-\frac{i}{5}}{3+\frac{3i z}{5}} & -\frac{2z-\frac{2}{3}}{3-z } \\
\end{array}
\right],\]
 and so, $M_{\widehat{T}_2}|_{H_3^2}$ is a minimal isometric dilation of $T_2$ and $M_{\widehat{T}_2}|_{L_3^2}$ is a minimal unitary dilation of $T_2$.

\end{example}

\section{Fractured Nelson Dilations} \label{fractured} 
In this section, we consider the structure and properties of (automorphic) Nelson dilations as in \eqref{eqn:nd}. We first prove Theorem \ref{thm:fractured}, which shows that every strictly contractive matrix has a Nelson dilation with a finite exceptional set $S_{\widehat{T}}$, as defined in \eqref{eqn:exceptional}. 

\begin{theorem*}\ref{thm:fractured}. Every square matrix $T$ with $\| T \| <1$ has a fractured Nelson dilation $\widehat{T}$. 
\end{theorem*}

\begin{proof} Let $T$ be an $n \times n$ matrix with singular value decomposition as in \eqref{eqn:svd}. Since $V^* U$ is invertible, Corollary \ref{cor:unitary} implies that there is a diagonal unitary matrix 
\[ \Omega = \begin{bmatrix} \omega_1 && \\
& \ddots & \\
&& \omega_n \end{bmatrix} 
\]
such that $V^* U \Omega$ has distinct eigenvalues. This implies 
\[ U \Omega V^* = V \left( V^* U \Omega \right)V^*\]
also has distinct eigenvalues. Let $\widehat{T}$ be a Nelson dilation as in \eqref{eqn:nd} such that the associated automorphisms $b_1, \dots, b_n$ satisfy
\[ b_j(0) = s_j \ \text{ and } \ \ b_j(1) = \omega_j.\]
In particular, one can take $b_j(z) = \lambda_j \left( \frac{ z+ a_j}{1+\bar{a}_jz} \right)$ with $\lambda_j = \bar{\omega}_j \left(\frac{\omega_j- s_j}{\bar{\omega}_j - \bar{s}_j} \right) \in \mathbb{T}$ and $a_j = s_j \bar{\lambda}_j \in \mathbb{D}$.

Then $\widehat{T}(1) = U \Omega V^* $ has distinct eigenvalues.  We claim that this forces $\widehat{T}$ to be a fractured Nelson dilation of $T$.  To see this, for $z \in \mathbb{D}_T$, consider the  polynomial
\begin{equation} \label{eqn:char} p(x,z) = \det \left( x I - \widehat{T}(z) \right) \prod_{j=1}^{n} \left ( 1+\bar{a}_j z\right)^n,\end{equation}
which is just the characteristic polynomial of $\widehat{T}(z)$ multiplied by a polynomial in $z$ that is nonvanishing in $\mathbb{D}_T$ and ensures $p(x,z)$ is actually a two-variable polynomial. Then a point $z \in \mathbb{D}_T$ is also in $S_{\widehat{T}}$ if and only if $p(\cdot, z)$ and its partial derivative $p_x (\cdot, z)$ share a common zero. We show that this can only happen for a finite number of $z \in \mathbb{D}_T$. 

To that end, write $p(x,z) =q(x,z) p_1(x,z)$ and $p_x(x,z) = q(x,z) r(x,z)$ such that $p_1, r$ share no common factors. Then since $\widehat{T}(1)$ has distinct eigenvalues, $q(x,1)$ must be constant. Since for every $z\in \mathbb{D}_T$, $p(x,z)$ has highest degree term $c x^n$ in $x$ with $c \ne 0$, $q(x,1)$ constant implies that $q(x,z)$ cannot have any $x$ dependence. Thus, $q(x,z) = q(z)$ has a finite number of zeros, call them $\hat{z}_1, \dots \hat{z}_d$ in $\mathbb{C}$.

Meanwhile, since $p_1$ and  $r$ share no common factors, Bezout's theorem (see \cite[Chapter 8.7]{cox07}) implies that  the zero sets of $p_1$ and $r$ in $\mathbb{C}^2$ intersect in (at most) a finite number of points, call them $( x_1, z_1), \dots, (x_k, z_k)$. Then $S_{\widehat{T}} \subseteq \{ \hat{z}_1, \dots \hat{z}_d, z_1, \dots, z_k \} \cap \mathbb{D}_T$, which gives the claim. 
\end{proof}

Before considering the fine structure of  fractured Nelson dilations,  we need some basic information about their eigenvalues and eigenprojections.  In what follows, we will use many of the definitions and facts from \cite[Chapter 2]{kato} about matrices whose entries are one-variable analytic functions. 

\begin{remark}  Let $\widehat{T}$ be an $n \times n$ fractured Nelson dilation, so that $S_{\widehat{T}}$ is finite.  Since the set of eigenvalues of $\widehat{T}(z)$, denoted $\sigma(\widehat{T}(z))$,  are the solutions to the algebraic equation $\det ( x I - \widehat{T}(z))=0$, they form an $n$-valued, analytic function on $\mathbb{D}_T$. Additionally, if $z_0 \in \mathbb{D}_T \setminus  S_{\widehat{T}}$ so that $\widehat{T}(z)$ has distinct eigenvalues, then there is a neighborhood $\mathcal{U}_{z_0}\subseteq \mathbb{D}_T$ of $z_0$ and analytic functions $\lambda_1, \dots, \lambda_n$ such that the eigenvalues of $\widehat{T}(z)$ are given by
 \begin{equation} 
 \label{eqn:eigenvalues} \lambda_1(z), \dots, \lambda_n(z), \quad \text{ for } z \in \mathcal{U}_{z_0}.
 \end{equation}
If $\gamma$ is a curve in $\mathbb{D}_T \setminus  S_{\widehat{T}}$, then the eigenvalues functions in \eqref{eqn:eigenvalues} can be analytically continued along $\gamma$. However, if $\gamma$ encloses any point(s) from $S_{\widehat{T}}$ in its interior, then after each circuit around such points, the values in \eqref{eqn:eigenvalues} may permute amongst themselves. Such a permutation is illustrated later in Figure \ref{fig:branches}. 
Meanwhile, if $z_0 \in S_{\widehat{T}}$, then the eigenvalues of $\widehat{T}(z)$ can be parameterized by $n$ continuous functions (actually Puiseux series) near $z_0$, but not necessarily by separately analytic functions.

To consider the eigenprojections of  $\widehat{T}(z)$, let $\Gamma$ be a simple, closed, positively-oriented curve in $\mathbb{D}_T$. If $ z_0\in \mathbb{D}_T$ and $\sigma(\widehat{T}(z_0)) \cap \Gamma = \emptyset,$ then
\begin{equation} \label{eqn:Pgamma}  P(z, \Gamma):= \frac{1}{2 \pi i} \int_\Gamma (x I - \widehat{T}(z) ) ^{-1} dx\end{equation}
is analytic in a neighborhood of $z_0$ and equals the sum of the eigenprojections associated to the eigenvalues of $\widehat{T}(z)$ in the interior of $\Gamma.$ If $z_0 \in \mathbb{D}_T \setminus S_{\widehat{T}}$, then nearby, the distinct eigenvalues of $\widehat{T}(z)$ are parameterized by \eqref{eqn:eigenvalues} and for each $j$, we can choose $\Gamma_j$ such that $\lambda_j(z_0)$ is the only element of $\sigma(\widehat{T}(z_0))$ in its interior. Then 
\[  P_{j}(z):= \frac{1}{2 \pi i} \int_{\Gamma_j} (x I - \widehat{T}(z) )^{-1} dx\]
gives the eigenprojection associated to eigenvalue $\lambda_j(z)$ for all $z$ in a  neighborhood $\mathcal{U}_{z_0}$ of $z_0$. By the residue theorem, we also have this alternative form for the eigenprojection:
\begin{equation}\label{eq:langrange} P_j(z) = \prod_{i\ne j} \left(\frac{ \widehat{T}(z) - \lambda_i(z) I }{\lambda_j(z) - \lambda_i(z)} \right)  \text{ for } z \in \mathcal{U}_{z_0},\end{equation}
which is well suited to computations.

Now fix some $k$ with $1 \le k \le n-1$, and assume that we have partitioned the eigenvalue function $\sigma(\widehat{T}(z))$ into a $k$-valued function $\Lambda_1(z)$ and an $(n-k)$-valued function $\Lambda_2(z)$ on $\mathbb{D}_T$. Then for $\ell=1,2$ and each $z\in \mathbb{D}_T \setminus S_{\widehat{T}}$, we can define the function
\begin{equation} \label{eqn:Plambda} P_{\Lambda_\ell}(z) = \sum_{j: \lambda_j(z) \in \Lambda_\ell(z)} P_{j}(z),\end{equation}
the sum of the eigenprojections associated to the eigenvalues of $\widehat{T}(z)$ in $\Lambda_\ell(z)$. 
\end{remark}

In order to identify a common invariant subspace for all of the $\widehat{T}(z)$, we will need a nontrivial eigenprojection to extend analytically  to a neighborhood of $\mathbb{D}$. That motivates the following definition, which roughly requires that both the eigenvalues of $\widehat{T}$ and their associated eigenprojections split nicely into two separate collections:

\begin{definition} \label{def:pr} A fractured Nelson dilation $\widehat{T}$ is called \textbf{projection reducible} if there exists a nontrivial partition of its eigenvalue function $\sigma(\widehat{T}(z))$ into a $k$-valued function $\Lambda_1(z)$ and an $(n-k)$-valued function $\Lambda_2(z)$ so that the projections $P_{\Lambda_1}$ and $P_{\Lambda_2}$ from \eqref{eqn:Plambda} extend analytically to $\mathbb{D}_T$. 
Otherwise, we say that $\widehat{T}$ is \textbf{projection irreducible}.
\end{definition} 

While the definition of \emph{projection reducible} might feel opaque, the following remark puts it in the context of other, perhaps more natural, ways of defining reducibility of $\widehat{T}$ via decompositions of its eigenvalue function and the associated eigenprojections.

\begin{remark} \label{rem:pr} 
We say that a fractured Nelson dilation $\widehat{T}$ is \textbf{totally eigenvalue reducible} if there is a nontrivial partition of its eigenvalue function $\sigma(\widehat{T}(z))$ into a $k$-valued function $\Lambda_1(z)$ and an $(n-k)$-valued function $\Lambda_2(z)$ such that both functions are continuous and $\Lambda_1(z) \cap \Lambda_2(z) = \emptyset$ for all $z\in \mathbb{D}_T.$ As mentioned in Section \ref{sec:intro}, if  $\widehat{T}$ is \emph{not totally eigenvalue reducible}, then we say $\widehat{T}$ is \textbf{eigenvalue entangled}.

%

If $\widehat{T}$ is totally eigenvalue reducible, then  $\widehat{T}$ must also be projection reducible.  Indeed,  for each $z_0 \in \mathbb{D}_T$, $\Lambda_1(z_0) \cap \Lambda_2(z_0) = \emptyset$ implies that there is a simple, closed, positively-oriented curve $\Gamma_{z_0} \subseteq \mathbb{D}_T$ such that $\Lambda_1(z_0)$ is in the interior of $\Gamma_{z_0}$ and $\Lambda_2(z_0)$ is in the exterior of $\Gamma_{z_0}$. Then because $\Lambda_1$ and $\Lambda_2$ are continuous functions, for $z$ sufficiently close to $z_0$, the elements of $\Lambda_1(z)$ must stay interior to $\Gamma_{z_0}$ and the elements of $\Lambda_2(z)$ must stay exterior to $\Gamma_{z_0}$. Then, the  eigenprojection $P_{\Lambda_1}(z)$ onto the eigenvalues in the set $\Lambda_1(z)$ can be represented locally by \eqref{eqn:Pgamma} with $\Gamma = \Gamma_{z_0}$. As this gives a formula for $P_{\Lambda_1}$ that is analytic in a neighborhood of each  $z_0 \in \mathbb{D}_T$, $P_{\Lambda_1}$ automatically has an analytic extension to $\mathbb{D}_T$. 

One could consider an alternative notion of reducibility, which we will call \textbf{eigenvalue reducible}, where the eigenvalue function $\sigma(\widehat{T}(z))$ is partitioned into continuous collections $\Lambda^*_1(z)$, $\Lambda^*_2(z)$  that are invariant under the eigenvalue permutations induced by analytical continuation along curves. With such a definition, one can show that the associated eigenprojection $P_{\Lambda^*_1}$ as in \eqref{eqn:Plambda} should extend analytically to $\mathbb{D}_T \setminus S_{\widehat{T}}$. However, such a $P_{\Lambda^*_1}$ could have poles at the exceptional points (a phenomenon that is explored in Example \ref{ex2}). \end{remark}

We can connect projection reducibility with the existence of a nontrivial reducing subspace of $\widehat{T}(z)$ that is independent of $z$. In particular, as mentioned in Section \ref{sec:intro}, we say that $\widehat{T}$ is \textbf{subspace reducible} if there is a subspace $S$ with $ \{0\} \subsetneq S \subsetneq \mathbb{C}^n$ such that for each $z \in \mathbb{D}_T$, $S$ is invariant under both $\widehat{T}(z)$ and $\widehat{T}(z)^*$. 
In what follows, we will need the following easy and well-known geometric lemma concerning eigenvectors and reducing subspaces.

\begin{lemma}\label{lem:ortho} 
Suppose that $A \in M_n(\mathbb C)$ has distinct eigenvalues and that $A$ has a reducing subspace $S$. Then, each eigenvector of $A$ is in $S$ or $S^\perp$. 
\end{lemma}
\begin{proof}
Suppose that $S$ is reducing for $A$. Then the orthogonal projection onto $S$, denoted $P_S$, commutes with $A$ (see e.g. \cite[II.3.7]{conway}). Now suppose that $\vec{v}$ is an eigenvector of $A$ associated to an eigenvalue $\lambda$. Then
\[A P_S \vec{v} = P_S A \vec{v} = \lambda P_S \vec{v}. \]
Hence $P_S \vec{v}$ is an eigenvector associated with $\lambda$ or $P_S \vec{v} = \vec{0}$, in which case $\vec{v} \in S^\perp$.  If $P_S \vec{v} \neq \vec{0}$, it must be that $P_S \vec{v} = \alpha \vec{v}$ for some $\alpha \in \mathbb{C} \setminus \{0\},$ as the eigenspace associated with $\lambda$ is one-dimensional, and so $\vec{v} \in S$.  Hence, each eigenvector $\vec{v}$ must lie in $S$ or $S^\perp$. 
\end{proof}
\color{black}

Now we can show that, for a fractured Nelson dilation, there is a straightforward connection between the concepts of projection reducible and subspace reducible.

\begin{theorem} \label{thm:constantPv} Let $T$ be an $n \times n$ matrix with $\| T \| <1$ and a fractured Nelson dilation $\widehat{T}$. Then $\widehat{T}$ is projection reducible if and only if $\widehat{T}$ is subspace reducible. \end{theorem}

\begin{proof} First assume $\widehat{T}$ is projection reducible. We will show that $\widehat{T}$ is subspace reducible. Let $\Lambda_1$ and $\Lambda_2$ be the associated partition functions of $\sigma(\widehat{T}(z))$.  Then $P_{\Lambda_1}$ as in \eqref{eqn:Plambda} extends to be analytic on $\mathbb{D}_T$ and so, $P_{\Lambda_1}$ extends to be bounded and analytic on a neighborhood of $\overline{\mathbb{D}}$. By construction, $\widehat{T}$ is unitary-valued on $\mathbb{T}$, so $\| P_{\Lambda_1}\| =1$ on $\mathbb{T}$. Meanwhile, since $P_{\Lambda_1}$ is projection-valued, $\|P_{\Lambda_1}(z) \| \ge 1$ for $z \in \mathbb{D}$. Then $P_{\Lambda_1}$ must attain its maximum norm on $\overline{\mathbb{D}}$ on the interior of the disk, so the maximum norm principle implies that $\|P_{\Lambda_1}\| \equiv 1$ on $\overline{\mathbb{D}}$, see e.g. \cite[Theorem 2]{condori}. 

A standard linear algebra fact says that if an $n \times n$ matrix $P$ is idempotent with $\| P\| =1$, then $P$ is Hermitian \cite[Proposition II.3.3]{conway}. Thus, $P_{\Lambda_1}(z)$ is Hermitian-valued on $\mathbb{D}$. Since the entries of $P_{\Lambda_1}$ are analytic on $\mathbb{D}_T$, this implies that they must be constant and so, $P_{\Lambda_1}$ is constant on $\mathbb{D}_T$.

Now, let $S \subseteq \mathbb{C}^n$ be the range of $P_{\Lambda_1}$. Then $S$ is invariant under $\widehat{T}(z)$ for each $z\in \mathbb{D}_T$. Indeed, if $z \in \mathbb{D}_T \setminus S_{\widehat{T}}$, then $S$ is the linear span of a finite set of eigenvectors of $\widehat{T}(z)$ and so, is clearly invariant under $\widehat{T}(z)$. Meanwhile if $z \in S_{\widehat{T}}$ and $\vec{x} \in S$, there is a sequence $(z_n)  \subseteq \mathbb{D}_T\setminus S_{\widehat{T}}$ with $(z_n) \rightarrow z$. Then each $\widehat{T}(z_n)\vec{x} \in S$ and so the continuity of $\widehat{T}$ paired with the fact that $S$ is closed in $\mathbb{C}^n$ implies $\widehat{T}(z)\vec{x} \in S$. 

Now if $z\in \mathbb{T}$, then $\widehat{T}(z)$ is unitary and so
\[ \widehat{T}(z) S \subset S \text{ implies that }  \widehat{T}(z) S= S,\]
and so $\widehat{T}(z)^* S =S$ as well. This implies that $P_{S} \widehat{T}(z) P_{S^\perp} =0$ for $z\in \mathbb{T}$, where $P_S, P_{S^{\perp}}$ denote the orthogonal projections onto $S, S^{\perp}$ respectively.  But by analyticity and the maximum principle, this immediately implies that $P_{S} \widehat{T}(z)P_{S^\perp} =0$ for all $z\in \mathbb{D}$ (and hence for all $z\in \mathbb{D}_T$) as well and so, each $\widehat{T}(z)^* S \subseteq S$. Thus, $S$ is an invariant subspace of both $\widehat{T}(z)$ and $\widehat{T}(z)^*$ for all $z\in \mathbb{D}_T$, so $\widehat{T}$ is subspace reducible.

For the converse direction, assume that $\widehat{T}$ has a reducing subspace $S$ with $\dim S =k$ for some $k \in \{1, \dots, n-1\}$. Then $S$ and  $S^{\perp}$ are invariant subspaces of $\widehat{T}(z)$ for all $z$ in $\mathbb{D}_T$. Fix $z \in \mathbb{D}_T \setminus S_{\widehat{T}}$. Then 
$\widehat{T}(z)$ has distinct eigenvalues as in \eqref{eqn:eigenvalues}. Let $\vec{v}_1(z), \dots, \vec{v}_n(z)$ be the associated eigenvectors. Then by Lemma \ref{lem:ortho}, since $\widehat{T}(z)$ has $n$ distinct eigenvalues and $S$ is reducing for $\widehat{T}(z)$, each $\vec{v}_i(z)$ must be in $S$ or $S^{\perp}$. Since $\dim S = k$, $S$ must be the span of $k$ eigenvectors of $\widehat{T}(z)$ and $S^\perp$ must be the span of the other $(n-k)$ eigenvectors of $\widehat{T}(z)$. Reordering if necessary, we can write $S = \text{span} \{  \vec{v}_1(z), \dots, \vec{v}_k(z) \}$ and $S^{\perp} = \text{span} \{  \vec{v}_{k+1}(z), \dots, \vec{v}_n(z) \}$. Define $\Lambda_1(z) = \{ \lambda_1(z), \dots, \lambda_k(z)\}$ and $\Lambda_2(z) =   \{ \lambda_{k+1}(z), \dots, \lambda_n(z)\}$. Then $\Lambda_1, \Lambda_2$ are defined on $\mathbb{D}_T \setminus S_{\widehat{T}}$ and we can define them on the finite set of points $z \in S_{\widehat{T}}$ in any way that partitions $\sigma(\widehat{T}(z))$. Let $P_S$ be the orthogonal projection onto $S$. Then 
 for all $z \in \mathbb{D}_T \setminus S_{\widehat{T}}$, $P_{\Lambda_1}(z)$ from \eqref{eqn:Plambda} equals $P_S$.  Thus, we can trivially extend $P_{\Lambda_1}$  analytically to $\mathbb{D}_T$ by defining it to be $P_S$ at each point. Thus, $\widehat{T}$ is projection reducible. 
 \end{proof} 

The following result is an immediate corollary of Theorem \ref{thm:constantPv} and Remark \ref{rem:pr}.

\begin{theorem*} \ref{thm:entangled}. Let $T$ be an $n\times n$ matrix with $\| T \| <1$ and fractured Nelson dilation $\widehat{T}$. Then at least one of the following must occur: $\widehat{T}$ is eigenvalue entangled or $\widehat{T}$ is subspace reducible.
\end{theorem*} 

\begin{proof} Assume that $\widehat{T}$ is not eigenvalue entangled. Then by definition, $\widehat{T}$ is totally eigenvalue reducible and so Remark \ref{rem:pr} implies that $\widehat{T}$ is projection reducible. By Theorem \ref{thm:constantPv}, $\widehat{T}$ is subspace reducible, as needed.
\end{proof}

We end this section by considering some examples.  First let us return to the matrices $T_1$ and $T_2$ and their respective Nelson dilations from Example \ref{ex1}.

\begin{example} \label{ex:fracture2} By computing the common zero set of each characteristic polynomial $p(x,z) = \text{det}( x I - \widehat{T}_j(z))$ and its $x$-derivative,  one can check that each of $\widehat{T}_1$ and $\widehat{T}_2$ have finite exceptional sets $S_{\widehat{T}_1}$ and $S_{\widehat{T}_2}$. Thus, these are fractured Nelson dilations. Indeed, their exceptional sets inside of $\mathbb{D}$ are plotted in Figure \ref{fig:exceptional}.

\begin{figure}[!ht] 
    \subfigure[The set $S_{\widehat{T}_1} \cap \mathbb{D}$]
      {\includegraphics[width=0.45 \textwidth]{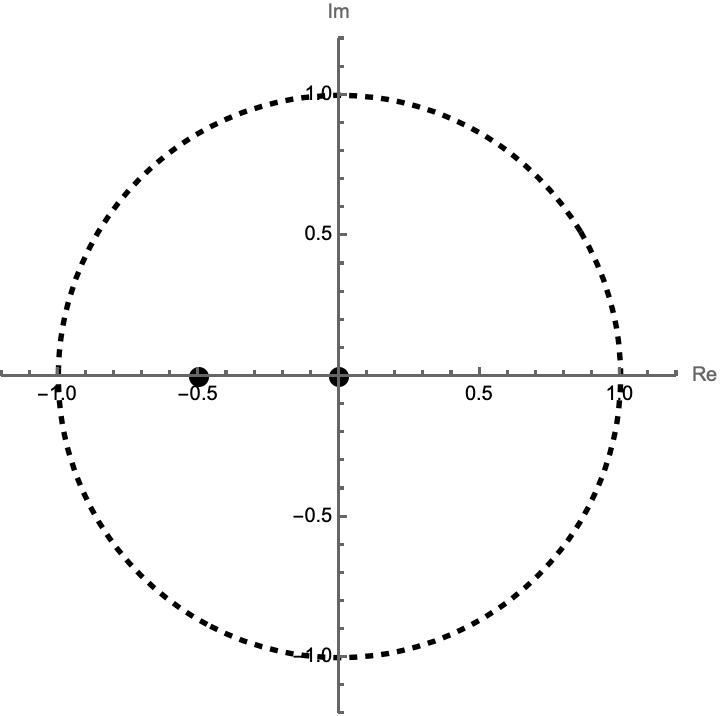}}
    \quad 
    \subfigure[The set $S_{\widehat{T}_2} \cap \mathbb{D}$]
      {\includegraphics[width=0.45 \textwidth]{t2B.png}}
  \caption{\textsl{The exceptional sets for $\widehat{T}_1$ and $\widehat{T}_2$ in $\mathbb{D}$ plotted with the unit circle.}}\label{fig:exceptional}
\end{figure}

It is also easy to check that neither $\widehat{T}_1(z)$ nor $\widehat{T}_2(z)$ have a common reducing subspace for all $z \in \mathbb{D}_T$. Indeed, if such a subspace $S$ existed for $\widehat{T}_j(z)$, Lemma \ref{lem:ortho} would imply that the eigenvectors of $\widehat{T}_j(z)$ would partition into two nontrivial collections (based on whether each is in $S$ or $S^\perp$) that are necessarily orthogonal. By checking specific $z$ values (e.g. $z=1/2$ for $\widehat{T}_1$ and $z=0$ for $\widehat{T}_2$), one can see that there are $z$-values where $\widehat{T}_j(z)$ does not have a nontrivial reducing subspace. Thus, both $\widehat{T}_1$ and $\widehat{T}_2$ fall into the settings of Theorem \ref{thm:constantPv} and Theorem \ref{thm:entangled}: they must be both projection irreducible and eigenvalue entangled. 
\end{example}

Now, let us explore the three notions of eigenvalue/eigenprojection reducibility from Remark \ref{rem:pr}. Here, we note that \textbf{projection reducible} does not imply \textbf{totally eigenvalue reducible}, even for fractured Nelson dilations. 

\begin{example} \label{ex:both} Specifically, consider
\begin{equation} \label{eqn:bothT}  \widehat{T}(z) = \begin{bmatrix} z & 0 \\ 0 & i z \end{bmatrix}.\end{equation}
Then $\widehat{T}$ is clearly a Nelson dilation of the zero matrix and is fractured because it has distinct eigenvalues for all $z \ne 0$. 
However, since $\widehat{T}$ is $2 \times 2$,  the repeated eigenvalue at $z=0$ means $\widehat{T}$ is not totally eigenvalue reducible. However,  $S = \text{span}\{ \vec{e}_1\}$ is a nontrivial reducing subspace for all $\widehat{T}(z)$ and so Theorem \ref{thm:constantPv} implies that $\widehat{T}$ is projection reducible (a fact that could also be proved directly).

This shows that the two properties in Theorem \ref{thm:entangled}-- eigenvalue entangled and subspace reducible-- are not mutually exclusive; there are fractured Nelson dilations $\widehat{T}$ like \eqref{eqn:bothT}  that satisfy both conditions.
\end{example}

Meanwhile, the following example shows that \textbf{eigenvalue reducible} does not imply \textbf{projection reducible}, even for fractured Nelson dilations. Thus, to get the biconditional statement in Theorem \ref{thm:constantPv}, we actually needed to use the notion of projection reducible instead of either of the other definitions from Remark \ref{rem:pr}.

\begin{example}\label{ex2}

Consider the function
  \[\widehat{T}(z) = \begin{bmatrix} 0& 0& 0& b(z) \\ z &0 &0 &0 \\ 0 & z & 0 & 0 \\ 0 & 0 & b(z) & 0 \end{bmatrix}, \]
 where $\displaystyle b(z) = \frac{z-\alpha}{1 - \overline{\alpha} z}$ for $0 < |\alpha| < 1$.  Then $\widehat{T}$ is a Nelson dilation as
 \[\widehat{T}(z) = \begin{bmatrix} 0&0&0&1 \\ 1 &0&0&0 \\ 0&1&0&0 \\ 0&0&1&0 \end{bmatrix} \begin{bmatrix} z &0&0&0 \\ 0&z&0&0 \\ 0&0&b(z)&0 \\ 0 &0&0&b(z) \end{bmatrix} I_4. \]
 Leveraging the structure of characteristic polynomials of permutation matrices, we compute 
 \[\det (\widehat{T}(z) - \lambda I) = (\lambda^2 - z b(z))(\lambda^2 + z b(z)). \]
Parametrize the local branches of the eigenvalue functions via $\lambda_1, \lambda_2 = \pm \sqrt{z b(z)}$ and $\lambda_3, \lambda_4 = \pm i \sqrt{z b(z)}.$  The eigenvalues thus form globally defined analytic 2-valued functions. Hence this $\widehat{T}$ is eigenvalue reducible. We can compute $P_{\{1,2\}}$ using the Lagrange-type formula in \eqref{eq:langrange} away from $z = 0$ and $z = \alpha$ as the eigenvalues are distinct. Then
\[P_{\{1,2\}}(z) = \frac{1}{2} \begin{bmatrix} 1 & 0 & \frac{b(z)}{z} & 0 \\ 0 & 1 & 0 & 1 \\ \frac{z}{b(z)} & 0 & 1 & 0 \\ 0 & 1 & 0 & 1 \end{bmatrix}.\]
So $P_{\{1,2\}}$ is analytic off the exceptional set, but unbounded at poles at $z = 0$ and $z = \alpha$, precisely in line with the discussion in Kato \cite[Chapter 2]{kato}.

We can also check directly that $\widehat{T}(z)$ is subspace irreducible. To see this, choose some $w \in \D$ so that $w \neq 0$, $b(w)\neq 0$ and $|w| \neq |b(w)|$. We will abuse notation for clarity and write $b(w) = b$ in the discussion below. For a fixed $w$, we have that the eigenvalues of $\widehat{T}(w)$ are the roots of the equation $\lambda^4 = w^2 b^2$. For each such $\lambda$, a corresponding eigenvector is given by 
\[ \vec{v}_\lambda = \begin{bmatrix}
  b/\lambda \\ wb/\lambda^2 \\ w^2 b / \lambda^3 \\ 1
\end{bmatrix}.\]

A straightforward calculation shows that when $\lambda, \mu$ are distinct,  $\vec{v}_\lambda \perp \vec{v}_\mu$ only when $|w| = |b|$, which violates our choice of $w$. Hence, the eigenvectors cannot be arranged into mutually orthogonal subcollections. Now observe that if $\widehat{T}(w)$ has a reducing subspace $S$, by Lemma \ref{lem:ortho} it must be that  $S$ is spanned by some subcollection of eigenvectors and $S^\perp$ must be spanned by the rest of them, which require two subcollections of mutually orthogonal eigenvectors. As this is impossible, $\widehat{T}(w)$ does not have a reducing subspace, and hence $\widehat{T}$ is not subspace reducible.  Then by Theorem \ref{thm:constantPv}, $\widehat{T}$ is projection irreducible.
\end{example}
\color{black}
\section{Diagonal Spectral Fracturing} \label{sec:fracturing}

In this section, we establish the key matrix analysis fact used in the proof of Theorem \ref{thm:fractured}, along with several related results. 

To start, let $D = D(d_1, \dots, d_n)$ be an $n\times n$ diagonal matrix with variables $d_1, \dots, d_n$ along the diagonal.  Let $M$ be a fixed $n \times n$ matrix and define the ($n+1$)-variable volume polynomial $V_M$ by 
\begin{equation} \label{eqn:VM} V_M(x, \vec{d}) = \det ( x I - M D),\end{equation}
where $I$ is always assumed to be of the appropriate size and $\vec{d}$ denotes $(d_1, \dots, d_n)$.  Then $V_M$ has degree $n$ in the variable $x$. Since each $d_j$ appears in exactly one column of the matrix $x I - M D$, $V_M$ has  degree at most $1$ in each $d_j$. When $x=0$, $V_M(0, \vec{d}) = (-1)^n( d_1 \cdots d_n) \det M$ and if we assume that $M$ is invertible (so $\det M \ne 0$), then $V_M$ must be exactly degree $1$ in each $d_j$.

Recall that a polynomial is \textbf{reducible} if it can be factored into a product of two nonconstant polynomials. In the following lemma we show that if $\det M \ne 0$ and if $V_M$ is reducible, then its factors must have a very specific form.

\begin{lemma} \label{lem:Vfact} Let $M$ be an invertible $n \times n$ matrix and assume that $V_M$ from \eqref{eqn:VM} is reducible. Then $V_M$ can be factored as 
\[ V_M(x, \vec{d}) = \det (x I- M_1 D_1) \det(x I- M_2 D_2),\]
where for some integer $\ell$ with $1 \le \ell \le n -1$, $M_1$ is an
$\ell \times \ell$ invertible matrix, $M_2$ is an $(n-\ell) \times (n-\ell)$ invertible matrix, $D_1$ is an $\ell \times \ell$ diagonal matrix, and $D_2$ is a $(n -\ell)\times (n-\ell)$ diagonal matrix. Moreover, the diagonal entries of $D_1$ and $D_2$ are exactly composed of the variables $d_1, \dots, d_n$.  
\end{lemma}

\begin{proof} Assume that $V_M$ is reducible. Then we can factor it as
\[ V_M(x, \vec{d}) = p(x, \vec{d}) q(x, \vec{d}),\] 
for nonconstant polynomials $p$ and $q$. Because $V_M$ has degree $1$ in each $d_i$, the polynomials $p$ and $q$ must depend on disjoint subsets of $\{ d_1, \dots, d_n\}$ and the union of those subsets must be exactly $\{ d_1, \dots, d_n\}$.

We claim that neither of those subsets can be empty. By contradiction, assume $p$ only depends on $x$. As the only term in $V_M$ that depends exclusively on $x$ is $x^n$, we have 
\[ x^n = V_M(x, 0 ) = p(x) q(x,0),\]
which implies that $p$ is of the form $c x^k$ for some $c\ne 0$ and $k >0$. Setting $x =0$ and $D =I$ gives $\det M =0$, so $M$ is not invertible, a contradiction. 

Thus, there is some integer $\ell$ with $1 \le \ell \le n -1$ and a permutation of $\{1,\dots, n\}$, which we denote $\{i_1, \dots, i_n\}$, such that $p$ depends on $\{d_{i_1}, \dots, d_{i_\ell}\}$ and $q$ depends on $\{d_{i_{\ell+1}}, \dots, d_{i_n}\}.$ Let $\tilde{D}$ be the $n\times n$ diagonal matrix with variables $d_{i_1}, \dots, d_{i_\ell}, d_{i_{\ell+1}}, \dots, d_{i_n}$ along the diagonal. Let $U$ be the $n\times n$ permutation matrix
\[ U_{jk} = \left\{ \begin{array}{ll} 1 & \text{ if } k ={i_j} \\
0 & \text{otherwise} \end{array} \right.. \]
Then $\tilde{D} = U D U^T$ and letting $\tilde{M} = U M U^T$ gives 
\begin{eqnarray}  \nonumber
\det (x I - \tilde{M} \tilde{D}) &= \det (x I - U M U^T U D U^T) = \det (x I - MD) \\ &= p(x, d_{i_1}, \dots, d_{i_\ell}) q(x, d_{i_{\ell+1}}, \dots, d_{i_n}). \label{eqn:VMfact}
\end{eqnarray}
Write $\tilde{M}$ as a block matrix in the form 
\[ \tilde{M} = \begin{bmatrix} M_1 & * \\
* & M_2 \end{bmatrix},\]
where $M_1$ is $\ell \times \ell$ and $M_2$ is $(n-\ell) \times (n-\ell)$. Let $D_1$ be the $\ell \times \ell$ diagonal matrix with entries $d_{i_1}, \dots d_{i_\ell}$ and $D_2$ be the $(n-\ell) \times (n-\ell)$ diagonal matrix with entries $d_{i_{\ell+1}}, \dots d_{i_n}.$ Note that setting $\vec{d} =0$ implies that 
\[ x^n = p(x,0) q(x,0),\]
and so, we can assume that $p(x,0)= x^k$ and $q(x,0) = x^{n-k}$ for some nonnegative integer $k$.  Then setting $d_{i_1}, \dots, d_{i_\ell} =0$ in \eqref{eqn:VMfact} gives
\begin{equation} \label{eqn:Vfact1} x^\ell \det (x I - M_2 D_2) = p(x,0) q(x, d_{i_{\ell+1}}, \dots, d_{i_n}) =x^kq(x, d_{i_{\ell+1}}, \dots, d_{i_n}). \end{equation}
Similarly, setting $d_{i_{\ell+1}}, \dots, d_{i_n}=0$ in \eqref{eqn:VMfact} gives
\begin{equation} \label{eqn:Vfact2} x^{n-\ell} \det (x I - M_1 D_1) = p(x, d_{i_1}, \dots, d_{i_\ell}) q(x, 0) =x^{n-k} p(x, d_{i_1}, \dots, d_{i_\ell}). \end{equation}
For $x \ne 0$, combining \eqref{eqn:Vfact1} and \eqref{eqn:Vfact2} with \eqref{eqn:VMfact} gives 
\begin{eqnarray}\nonumber
 \det (x I &-& MD) = p(x, d_{i_1}, \dots, d_{i_\ell}) q(x, d_{i_{\ell+1}}, \dots, d_{i_n}) \\
 & =& \frac{x^{n-\ell} \det (x I - M_1 D_1) }{x^{n-k} } \frac{ x^\ell \det (x I- M_2 D_2)}{x^k} =  \det (x I- M_1 D_1) \det (x I - M_2 D_2). \label{eqn:Vfact3} 
\end{eqnarray}
Since both sides of \eqref{eqn:Vfact3} are polynomials, the formula clearly holds when $x=0$ as well. Lastly, substituting $D=I$ and $x=0$ into \eqref{eqn:Vfact3} gives $\det M =  \det M_1 \det M_2$, so $M_1$  and $M_2$ must both be invertible.  
\end{proof}

Lemma \ref{lem:Vfact} tells us something about the allowable roots of $V_M$ or equivalently, the eigenvalues of $MD$.

\begin{theorem} \label{thm:Vroots}  Let $M$ be an invertible $n\times n$ matrix. For almost every $\vec{d} \in \mathbb{C}^n$, the polynomial $V_M(\cdot, \vec{d})$ has distinct zeros. 
\end{theorem} 

\begin{proof} By applying Lemma \ref{lem:Vfact} iteratively, we can write 
\begin{equation} \label{eqn:VMirr} V_M(x, \vec{d}) = \prod_{k=1}^K \det (x I - M_k D_k), \end{equation}
for invertible matrices $M_k$ whose sizes add to $n$, diagonal matrices $D_k$  whose entries combine to give the variables $d_1, \dots, d_n$, and some $K \le n$. We can further assume that each $\det (x I - M_k D_k)$ is irreducible.

Proceeding by contradiction, assume that there is a set of positive measure $\mathcal{U} \subseteq \mathbb{C}^n$ such that for each $\vec{d}_0 \in \mathcal{U}$, $V_M(\cdot, \vec{d}_0)$ has a double root at $x_0$ for some $x_0 \in \mathbb{C}$. This implies that 
\begin{equation} \label{eqn:tildeV}  \tilde{V}_M (x, \vec{d}) := \frac{\partial V_M}{\partial x}(x, \vec{d})  = \sum_{j=1}^K \frac{\partial}{\partial x}\Big( \det ( x I - M_j D_j) \Big ) \prod_{k \ne j} \det (x I - M_k D_k)\end{equation}
also vanishes at each $(x_0, \vec{d}_0)$. 

Let $p$ be the resultant of $V_M$ and $\tilde{V}_M$ with respect to $x$. Note that our assumption about repeated zeros forces $n \ge 2$  and thus, $V_M$ and $\tilde{V}_M$ have positive degree in $x$. Then we are free to apply standard facts about resultants of multivariate polynomials, see Chapter 3, Section 6 in the textbook \cite{cox07} and in particular Proposition 1 on page 163. Specifically, the resultant $p$ must be in both the ideal  $\langle V_M, \tilde{V}_M \rangle$ and the polynomial ring $\mathbb{C}[d_1, \dots, d_n].$ 
Thus, for all $\vec{d}_0 \in \mathcal{U}$, since there is some  $x_0 \in \mathbb{C}$ with $V_M(x_0, \vec{d}_0) =0= \tilde{V}_M(x_0, \vec{d}_0),$  we must have $p(\vec{d}_0) =0$ as well.
Since $\mathcal{U}$ is a set of positive measure in $\mathbb{C}^n$, this implies that $p \equiv 0$ and thus, $V_M$ and $\tilde{V}_M$ have a common factor $q$ with positive degree in $x$.

By considering one of its factors if necessary, we can assume that $q$ is irreducible. \color{black} Because we have factored $V_M$ into irreducible factors, we can further assume that
\[ q(x, \vec{d}) =\det (x I  -M_{j_0} D_{j_0})\]
for some $j_0$ with $ 1 \le j_0 \le K$. Then $ \det (x  I -M_{j_0} D_{j_0})$ must be a factor of $\tilde{V}_M$ as given in \eqref{eqn:tildeV} and in particular, it must be a factor of the specific term in the sum
\[ r(x, \vec{d}):=\frac{\partial}{\partial x}\Big( \det ( x I - M_{j_0} D_{j_0}) \Big ) \prod_{k \ne j_0} \det (x I - M_k D_k) = \frac{\partial q}{\partial x}( x, \vec{d}) \prod_{k \ne j_0} \det (x I - M_k D_k). \]
Since all of the factors in \eqref{eqn:VMirr} are irreducible and distinct (since they depend on different sub-collections of $\{d_1, \dots, d_n\}$), $q$ cannot divide any polynomial of the form $\det (x I - M_k D_k)$ with $k \ne j_0$. Thus, $q$ must divide $ \frac{\partial q}{\partial x}$, which contradicts the fact that $q$ has positive degree in $x$ and finishes the proof. 
\end{proof}

The following is an immediate corollary of Theorem \ref{thm:Vroots}.

\begin{corollary} \label{cor:MD}  Let $M$ be an invertible $n\times n$ matrix.  Then for almost every diagonal matrix $D$, $MD$ has distinct eigenvalues. 
\end{corollary}

Recall that  $\mathbb{T}$ denotes the unit circle $\mathbb{T} = \{ z \in \mathbb{C}: |z| =1\}$. The following is basically a corollary of the proof of Theorem \ref{thm:Vroots}.   

\begin{theorem} \label{thm:Vroots2}   Let $M$ be an invertible $n\times n$ matrix. For almost every  $\vec{\omega} \in \mathbb{T}^n$, the polynomial $V_M(\cdot, \vec{\omega})$ has distinct zeros. 
\end{theorem} 

\begin{proof}  The proof is basically the same as that of Theorem \ref{thm:Vroots}. We just need the following fact:  if a polynomial $p \in \mathbb{C}[z_1, \dots, z_n]$ is zero on a set of positive measure in $\mathbb{T}^n$, then $p$ is identically zero. This is clearly true if we replace $\mathbb{T}^n$ with $\mathbb{R}^n$. However, $\mathbb{T}$ (with a point removed) and $\mathbb{R}$ are conformally equivalent via maps such as $\beta(z) =i \left(\frac{1 -z}{1+z}\right).$  Thus, one basically just needs to compose with a conformal map (and simplify) to conclude the desired fact. 

Now, assume by contradiction that there is a set of positive measure $\mathcal{U} \subseteq \mathbb{T}^n$ such that for each $\vec{\omega}_0 \in \mathcal{U}$, $V_M(\cdot, \vec{\omega}_0)$ has a double root at $x_0$ for some $x_0 \in \mathbb{C}$. Then one just follows the same proof as in Theorem \ref{thm:Vroots},  using the above fact to conclude the resultant $p$ of $V_M$ and $\tilde{V}_M$ is identically zero.
\end{proof}

Theorem \ref{thm:Vroots2} has this immediate corollary, which we used in the proof of Theorem \ref{thm:fractured}:

\begin{corollary} \label{cor:unitary}  Let $M$ be an invertible $n\times n$ matrix. For  almost every $n \times n$ diagonal unitary $\Omega$, $M \Omega$ has distinct eigenvalues. 
\end{corollary} 

\section{Open Questions}\label{sec:open}

This investigation motivates three related open questions. First, for each $n \times n$ matrix $T$ with $\|T \| <1$, we now know that there is a Nelson dilation $\widehat{T}$ of $T$ whose exceptional set $S_{\widehat{T}}$ is finite. However, it is not clear how the sets $S_{\widehat{T}}$ behave as we increase the size of the matrices under consideration. This motivates the question:

\begin{question}  \emph{As $n \rightarrow \infty$, is there a ``law of large numbers'' that explains how  the exceptional sets $S_{\widehat{T}}$ accumulate in $\overline{\mathbb{D}}$?}  \end{question}

If the points in the exceptional sets $S_{\widehat{T}}$   occur with an absolutely continuous distribution supported randomly on the disk, then this provides an obstruction to approaching the invariant subspace problem using the tools in this paper. In particular, one cannot use Montel's theorem to continue the subspaces that exist for points on the boundary (where unitary operators occur) inside the disk. This has a similar flavor to Jentzsch's theorem, which states that the partial sums of power series have zeros that accumulate everywhere on the boundary of the disk of convergence, instead of just on the arcs where the underlying function does not analytically continue.

Secondly, in our study of fractured Nelson dilations, we noted that if $z_0 \in \mathbb{D}_T \setminus  S_{\widehat{T}}$, then there is a neighborhood $\mathcal{U}_{z_0}$ of $z_0$ where the eigenvalue function $\sigma(\widehat{T}(z))$ of $\widehat{T}(z)$ can be parameterized by analytic functions as in \eqref{eqn:eigenvalues}. While these eigenvalues functions can be analytically continued along curves $\gamma \subseteq \mathbb{D}_T \setminus  S_{\widehat{T}}$, the eigenvalue functions \eqref{eqn:eigenvalues} might permute if $\gamma$ encloses any point(s) in $S_{\widehat{T}}$.   

This is illustrated in Figure \ref{fig:branches}, which uses the fractured Nelson dilation $\widehat{T}_2$ from Examples \ref{ex1} and \ref{ex:fracture2}. On the left, we have chosen a curve $\gamma$ in $\D$ that encloses several points in $S_{\widehat{T}}$. By starting with a specific parameterized eigenvalue function  in a neighborhood of a point on $\gamma$ (say $\lambda_1(z)$ in a neighborhood of $z_0 \in \gamma$) and analytically continuing around $\gamma$ once, we end at a different eigenvalue function, say $\lambda_2(z)$ from the original parameterization of  $\sigma(\widehat{T}_2(z))$ a neighborhood of $z_0$. Continuing to analytically continue around $\gamma$ twice more, we run through the remaining branches of $\sigma(\widehat{T}_2(z))$ and end back at $\lambda_1(z).$ This path of three connected eigenvalue functions is graphed in  Figure \ref{fig:branches} on the right.
 
\begin{figure}[!ht] 
    \subfigure[A loop $\gamma$ in $\D$]
      {\includegraphics[width=0.45 \textwidth]{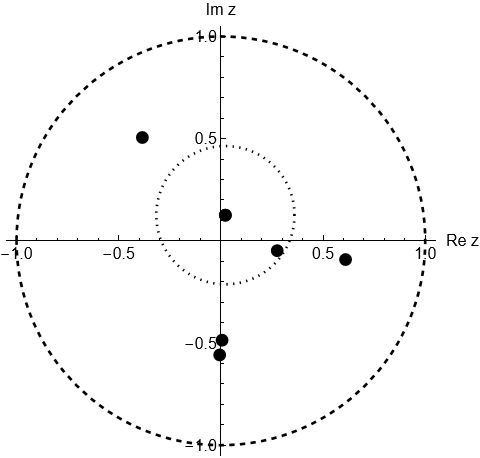}}
    \quad 
    \subfigure[Paths of the eigenvalues of $\widehat{T}_2$ on $\gamma$.]
      {\includegraphics[width=0.45 \textwidth]{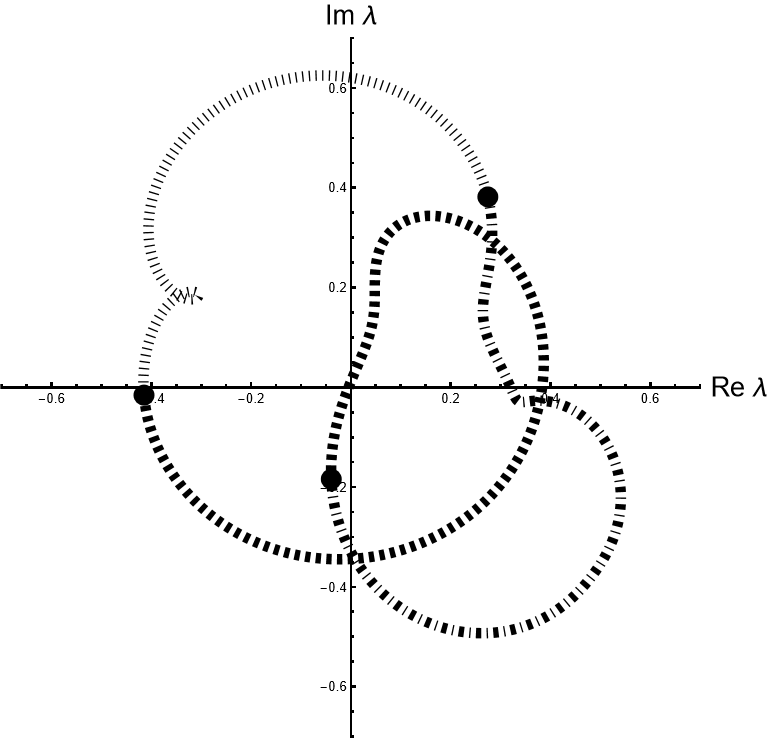}}
  \caption{\textsl{The eigenvalues of $\widehat{T_2}$ permute as we traverse $\gamma$.}}\label{fig:branches}
\end{figure}

Thus, every loop in $\D_T \backslash S_{\widehat{T}}$ can be viewed as inducing a permutation on the eigenvalues of $\widehat{T}$. This motivates the following two questions:

\begin{question}
\emph{Which subgroups of $S_n$ (the group of all permutations on $n$ symbols) can arise as monodromy groups of fractured Nelson dilations? For a generic Nelson dilation, is the monodromy group all of $S_n$?}
\end{question}

Lastly,  Lemma \ref{lem:Vfact} shows that if the volume polynomial $V_M$ from \eqref{eqn:VM} is reducible as a polynomial, then its factors are also volume polynomials of matrices that are related to the original matrix $M$. But, it is not clear whether the reducibility of $V_M$ forces $M$ to actually be reducible as a matrix.  This motivates the question:

\begin{question} \emph{For which $n \times n$ matrices $M$ does the volume polynomial $V_M$ factor as in Lemma \ref{lem:Vfact}?} \end{question}

If $M$ is reducible as an $n \times n$ matrix, then by definition, $M = U^T \tilde{M} U$, where $U$ is a permutation matrix  and $\tilde{M}$ is a block upper triangular matrix of form:
\begin{equation} \label{eqn:blockM2} \tilde{M} = \begin{bmatrix} M_1 & * \\
0 & M_2 \end{bmatrix}.\end{equation}
In this case, one can define $\tilde{D} = U D U^T= D_1 \oplus D_2$ for some diagonal matrices $D_1, D_2$ of the variables in $\vec{d}$ and as in \eqref{eqn:VMfact} and the proof of Lemma \ref{lem:Vfact},
\[ V_M(x, \vec{d}) = \det (x I - U M U^T U D U^T) = \det (x I - \tilde{M} \tilde{D}) =  \det (x I -M_1 D_1)   \det (x I -M_2 D_2).\]
Meanwhile, if $M$ is singular, then $MD$ will be singular for any choice of $\vec{d}$ and so, $V_M$ will always have a factor of $x$. (In Lemma \ref{lem:Vfact}, we required $M$ to be nonsingular to avoid this separate factor of $x$.) 

However if $M$ is irreducible and nonsingular, then the factorability of $V_M$ is not clear.   A natural concrete question is the following: Does there exist an irreducible, nonsingular $n \times n$ matrix $M$ whose volume polynomial $V_M$ is reducible?

\bibliographystyle{acm} 

\begin{thebibliography}{12}


\bibitem{CHLS12}
Man-Duen Choi, Zejun Huang, Chi-Kwong Li, and Nung-Sing Sze, 
Every invertible matrix is diagonally equivalent to a matrix with distinct eigenvalues.
\emph{Linear Algebra Appl.} \textbf{436} (2012), no. 9, 3773–3776.

\bibitem{condori}
Alberto Condori,
Maximum principles for matrix-valued analytic functions.
\emph{Amer. Math. Monthly} \textbf{127} (2020), no. 4, 331–343.

\bibitem{conway}
John B. Conway,
\emph{A Course in Functional Analysis}, 2nd ed.
Graduate Texts in Mathematics, vol. 96. Springer-Science. New York. 2007.

\bibitem{cox07}
David Cox, John Little, Donal O'Shea, 
\emph{Ideals, Varieties, and Algorithms.
An introduction to computational algebraic geometry and commutative algebra.} Third edition.
Undergrad. Texts Math.
Springer, New York, 2007.

\bibitem{FLH12}
Xin-Lei Feng,  Zhongshan Li, Ting-Zhu Huang, 
Is every nonsingular matrix diagonally equivalent to a matrix with all distinct eigenvalues?.
\emph{Linear Algebra Appl.} \textbf{436} (2012), no. 1, 120–125.

\bibitem{garciamashreghiross}
Stephan Ramon Garcia, Javad Mashreghi, and William T. Ross, \emph{Introduction to Model Spaces and their Operators},  Cambridge Studies in Advanced Mathematics, vol. 148, Cambridge University Press, Cambridge, 2016.


\bibitem{hartz} Michael Hartz, 
Von Neumann's inequality for commuting weighted shifts.
\emph{Indiana Univ. Math. J.} \textbf{66} (2017), no. 4, 1065–1079.

\bibitem{halmos} Paul Halmos,
Normal dilations and extensions of operators.
\emph{Summa Brasil. Math.}, \textbf{II} (1950) 125--134 

\bibitem{kato} Tosio Kato,
\emph{Perturbation Theory for Linear Operators.}
Classics in Mathematics. vol. 132, Springer-Verlag. Berlin. 1980.

\bibitem{mp} Scott McCullough and J. E. Pascoe,
Geometric dilations and operator annuli.
\emph{J. Funct. Anal.} \textbf{285} (2023), no. 7, Paper No. 110035, 20 pp.


\bibitem{nagy53}
Sz.-Nagy,
Sur les contractions de l'espace de Hilbert.(French)
Acta Sci. Math. (Szeged) 15 (1953), 87–92.

\bibitem{nagyfoias2010}
B\'ela  Sz.-Nagy, Ciprian Foias, Hari Bercovici, and  L\'aszl\'o K\'erchy,
\emph{Harmonic Analysis of Operators on Hilbert Space.}
Second edition. Universitext Springer, New York, 2010. 



\bibitem{nelson}  Edward Nelson,
The distinguished boundary of the unit operator ball.
\emph{Proc. Amer. Math. Soc.} \textbf{12} (1961), 994–995.

\bibitem{pall}
Gordon Pall, 
On the rational automorphs of $x^2_1+x^2_2+x^2_3$.
\emph{Ann. of Math.} (2) \textbf{41} (1940), 754--766.


\bibitem{paulsen} Vern Paulsen, \emph{Completely Bounded Maps and Operator Algebras}. 
Cambridge Stud. Adv. Math., vol 78
Cambridge University Press, Cambridge, 2002.

\bibitem{pisier} Giles Pisier, \emph{Similarity Problems and Completely Bounded Maps}, 2nd ed. Lecture Notes in Mathematics. vol 618, Springer-Verlag, Berlin, 2001.



\bibitem{shalit21} 
Orr Shalit,
Dilation theory: a guided tour. \emph{Operator theory, functional analysis and applications}, 551–623.
\emph{Oper. Theory Adv. Appl.}, \textbf{282}
Birkhäuser/Springer, Cham, 2021. 

\bibitem{schaffer55}
J. J. Sch\"affer,
On unitary dilations of contractions.
\emph{Proc. Amer. Math. Soc.} \textbf{6} (1955), 322.

\end{thebibliography}

\end{document}